\documentclass[12pt]{article}
\usepackage[amsmath]{e-jc}

\usepackage[utf8]{inputenc}
\usepackage[T1]{fontenc}
\usepackage{amssymb,latexsym}
\usepackage{amsmath}
\usepackage{amsthm}
\usepackage{thmtools}
\usepackage{graphicx}
\usepackage{xcolor}
\usepackage[shortlabels]{enumitem}
\usepackage{cleveref}

\usepackage{listings}%
\usepackage{esvect} 
\usepackage{todonotes}
\usepackage{subcaption}
\usepackage{url}

\usepackage{diagbox}

\def\RT{\mathrm{RT}}

\def\N{\mathbb N}
\def\A{\mathcal A}
\def\B{\mathcal B}

\def\uu{\mathbf{u}}

\dateline{TBD}{TBD}{TBD}
\Copyright{The authors. Released under the CC BY-ND license (International 4.0).}

\title{The limit of repetition thresholds of rich sequences}
\author{Lubom\'ira Dvo\v r\'akov\'a\authornote{1}
\and Edita Pelantov\'a\authornote{1}
}
\authortext{1}{FNSPE Czech Technical University, Prague, Czech Republic.
\href{mailto:lubomira.dvorakova@fjfi.cvut.cz}{\tt lubomira.dvorakova@fjfi.cvut.cz}.}

\begin{document}

\maketitle

\begin{abstract}
The repetition threshold of a class of sequences is the smallest number $r$ such that a~sequence from the class contains no repetition with exponent $> r$. We focus on the class~$\mathcal{C}_d$ of $d$-ary sequences rich in palindromes. In 2020, Currie, Mol, and Rampersad determined the repetition threshold for $\mathcal{C}_2$. In 2024, Currie, Mol and Peltom\"aki found the repetition threshold for $\mathcal{C}_3$ and conjectured that the repetition threshold for $\mathcal{C}_d$  tends to~2 with  $d$ growing to infinity. Here we verify their conjecture. 
\end{abstract}

\section{Introduction}
A~sequence is called {\em rich (in palindromes)} if each of its factors contains the maximum number of distinct palindromes. 
The definition was motivated by an observation by Justin et al.~\cite{DrJuPi2001} that the number of distinct non-empty palindromic factors contained in a word of length $n$ is smaller than or equal to $n$. The terminology was introduced by Glen et al. in 2009~\cite{GlJuWi2009}.
Rich sequences may be characterized in many different ways: using complete return words~\cite{GlJuWi2009} or using a relation between factor and palindromic complexity~\cite{BuLuGlZa2009} or using extensions of bispecial factors~\cite{BaPeSt2010}.

In this paper, we focus on repetitions in rich sequences. Any prefix $z$ of length $p$ of $u^\omega=uuu\cdots$ can be written as $u^e$, where $e=\frac{p}{q}$ and $q$ is the length of $u$. If $u$ is of minimum length, $e$ is called the {\em exponent} of $z$. 
For example, the word $magma$ can be written as ${mag}^{\frac{5}{3}}$.

\medskip
The {\em critical exponent}  $E(\uu )$ of a sequence $\uu$ is defined as
$$E(\uu) =\sup\{e \in \mathbb{Q}: \  e \  \text{is the exponent of a non-empty factor of $\uu$}\}\,.
$$
In addition, the {\em asymptotic critical exponent} $E^*(\uu)$, introduced by Cassaigne~\cite{Ca2008} under the name asymptotic index, 
is equal to $+\infty$,  if $E(\uu) = +\infty$, otherwise 
 $$E^*(\uu) =\lim_{n\to \infty}\sup\{e \in \mathbb{Q}: \  e \  \text{is the exponent of a factor of $\uu$ of length} \geq  n  \}\,.$$
It is obvious that $E^*(\uu)\leq E(\uu)$.

An important characteristic of a class of sequences is the infimum of critical exponents, resp. of asymptotic critical exponents, of sequences from this class.
We call this parameter {\em repetition threshold} for $C$ 
$$\RT(C) = \inf\{E(\uu): \uu \ \text{\ belongs to the class\ } C\}\,,$$
resp. {\em asymptotic repetition threshold} for $C$
$$\RT^*(C) = \inf\{E^*(\uu): \uu \ \text{\ belongs to the class\ } C\}\,.$$
Obviously, $\RT^*(C)\leq \RT(C)$. 

In this paper, we focus on the class $\mathcal{C}_d$ of $d$-ary sequences rich in palindromes.
For any rich sequence $\uu$, the inequality $E(\uu)\geq E^*(\uu)\geq 2$ holds; see~\cite{PStarosta2013}.

Consequently, we have a~lower bound
$$\RT(\mathcal{C}_d)\geq 2\,.$$

The repetition threshold of rich sequences has so far been determined only for binary and ternary sequences. The value of the repetition threshold for the binary alphabet  $ \RT(\mathcal{C}_2) =2+\frac{\sqrt{2}}{2} \approx 2.707\,$ was conjectured in \cite{BaSh19}, and proved by Currie, Mol, and Rampersad \cite{CuMoRa2020}. 
Recently, Currie, Mol, and Peltom{\"a}ki~\cite{CuMoPe2024} have shown that
$\RT(\mathcal{C}_3)=1 + \frac{1}{3-\mu}\approx 2.259$, where $\mu$ is the unique real root of the polynomial $x^3-2x^2-1$.

Currie, Mol, and Peltom{\"a}ki formulated in \cite{CuMoPe2024}:
\begin{conjecture}\label{con:jejichConjecture} 
    $\lim\limits_{d\to\infty} \RT(\mathcal{C}_{d}) =2$.
\end{conjecture}

Since ${ \RT}(\mathcal{C}_{d+1}) \leq {\RT(\mathcal{C}_d)} $, it is clear that the limit $\lim\limits_{d\to\infty} \RT(\mathcal{C}_{d}) $ exists. In this paper, we confirm the validity of their conjecture.

\medskip
In collaboration with Klouda~\cite{DvKlPe2025}, we defined a~$d$-ary sequence $\uu_d$ for every $d \in \N$, $d\geq 3$ as a fixed point of the morphism 
\begin{equation}\label{eq:morphismFi}\varphi_{d}:\left\{\begin{array}{rrcl} 
&  0 & \to &   01 \\
&  1 & \to &   02 \\
&   2 & \to &   03 \\
& ~\vdots\, && ~\vdots \\
&   d-2 & \to &   0(d-1) \\
&   d-1 & \to &   0(d-1)(d-1)\,. 
\end{array}\right.\end{equation}

We used these sequences to show that the asymptotic repetition threshold for the class of rich sequences does not depend on the size of the alphabet. In particular, we proved that $$\RT^*(\mathcal{C}_{d}) =2= \lim\limits_{n\to\infty}E^*(\uu_n) \quad \text{for every $d \in \N, d\geq 2$}. $$

In this paper, we prove the following theorem. 

   \begin{theorem}\label{thm:hlavni} For every $d \in \N, d\geq 3$, there exists a morphism $\pi: \{0,1,\dots, 2d\}^* \to {\mathcal B}^*$ with $\#\mathcal B=d+2^d$ such that the sequence  $\pi(\uu_{2d+1})$ is rich and \begin{equation}\label{eq:vHlavnim}
   E\bigl(\pi(\uu_{2d+1})\bigr) = E^*(\uu_{2d+1}).\end{equation}
\end{theorem}

The above theorem implies that $\RT({\mathcal C}_{d+2^d})\leq E^*(\uu_{2d+1})$, thus the validity of Conjecture \ref{con:jejichConjecture} is a~straightforward consequence.

\medskip

Let us emphasize an interesting feature of the sequences $\uu_{3}$ and $\uu_5$, namely 
\begin{equation}\label{eq:zname}
\RT(\mathcal{C}_2) = E^*(\uu_3)\qquad \text{and}\qquad RT(\mathcal{C}_3) = E^*(\uu_5).
    \end{equation}
In addition, the estimate $2.117< \RT(\mathcal{C}_4) < 2.12$  announced in \cite{CuMoPe2024} and the value $E^*(\uu_{7}) \doteq 2.119 71968$ together with \eqref{eq:zname} suggest that 
\begin{conjecture}\label{con:opet}  
   $\RT(\mathcal{C}_{d}) = E^*(\uu_{2d-1}) $ for every $d \in \N$, $d\geq 2$.
\end{conjecture}  
Determining the value $\RT(\mathcal{C}_{d})$ for the parameter $d=3$ required intensive computer support. Therefore, it seems that a proof of Conjecture \ref{con:opet} for a~general $d$ is currently beyond our capabilities.

\medskip

The paper is organized as follows. In Preliminaries, we recall the notions and theorems we use. In Sections \ref{sec:palindrome}
 and \ref{sec:nonpalindrome}, we derive properties of palindromic and non-palindromic  bispecial factors of $\uu_{2d+1}$.  The morphism $\pi$ which plays an important role in Theorem \ref{thm:hlavni} is introduced in Section \ref{sec:pi}, where we also demonstrate that the image of $\uu_{2d+1}$ under this morphism is rich. The proof of Theorem \ref{thm:hlavni}  itself is the content of Section \ref{sec:proofThm}. In Section \ref{sec:comments}, we discuss properties of other morphisms that could lead to a proof of Conjecture \ref{con:opet}. The appendix contains simple facts on the eigenvalues of the matrix of the morphism fixing $\uu_{2d+1}$.

\section{Preliminaries}
An \textit{alphabet} $\mathcal{A}$ is a finite set of symbols, called \textit{letters}. 
A \textit{word} $u$ over $\mathcal A$ is a finite string of letters from $\mathcal A$. The length of $u$ is denoted $|u|$ and $|u|_{i}$ denotes the number of occurrences of the letter ${i}$ in the word $u$. The \textit{Parikh vector} $ \vec{u}$ is the vector whose $i$-th component equals $|u|_{i}$. The set of all finite words over $\A$ is denoted $\A^*$. The set $\A^*$ equipped with concatenation as the operation forms a monoid with the \textit{empty word} $\varepsilon$ as the neutral element. Consider $u, p, s, v \in \A^*$ such that $u=pvs$, then the word $p$ is called a \textit{prefix}, the word $s$ a \textit{suffix} and the word $v$ a \textit{factor} of $u$. 

A~\textit{sequence} $\uu$ over $\A$ is an infinite string $\uu = u_0 u_1 u_2 \cdots$ of letters $u_j \in \A$ for all $j \in \N$. A \textit{word} $v$ over $\mathcal A$ is called a~\textit{factor} of the sequence $\uu = u_0 u_1 u_2 \cdots$ if there exists $j \in \mathbb N$ such that $v= u_j u_{j+1} u_{j+2} \cdots u_{j+|w|-1}$. The integer $j$ is called an \textit{occurrence} of the factor $v$ in the sequence $\uu$. If $j=0$, then $v$ is a \textit{prefix} of $\uu$.

The set of factors occurring in a~sequence $\uu$ forms the \textit{language} $\mathcal{L}(\uu)$ of $\uu$.
The language $\mathcal{L}(\uu)$ is said to be \textit{closed under reversal} if for each factor $v=v_0v_1\cdots v_{n-1}$, its \textit{mirror image} $\overline{v}=v_{n-1}\cdots v_1 v_0$ is also a factor of $\uu$.
A~factor $w$ of a sequence $\uu$ is \textit{left special} if $iw, jw \in \mathcal{L}(\uu)$ for at least two distinct letters ${i, j} \in \A$. A \textit{right special} factor is defined analogously. A factor is called \textit{bispecial} if it is both left and right special. The set of left extensions is denoted $\mathrm{Lext}(w)$, i.e., $\mathrm{Lext}(w)=\{iw \in {\mathcal L}(\uu)\ : \ i \in {\mathcal A}\}$. Similarly, $\mathrm{Rext}(w)=\{wi \in {\mathcal L}(\uu) \ :\ i \in {\mathcal A}\}$. The set of both-sided extensions is denoted $\mathrm{Bext}(w)=\{iwj \in {\mathcal L}(\uu)\ :\ i,j \in {\mathcal A}\}$. The \textit{bilateral order} $\mathrm{b}(w)$ of $w\in {\mathcal L}(\uu)$ is defined 
$$\mathrm{b}(w)=\#\mathrm{Bext}(w)-\#\mathrm{Lext}(w)-\#\mathrm{Rext}(w)+1\,.$$

A sequence $\uu$ is \textit{recurrent} if each factor of $\uu$ has at least two occurrences in $\uu$. Moreover, a recurrent sequence $\uu$ is \textit{uniformly recurrent} if the distances between consecutive occurrences of each factor in $\uu$ are bounded. If a uniformly recurrent sequence $\uu$ contains infinitely many palindromic factors, then its language ${\mathcal L}(\uu)$ is closed under reversal. 
A~sequence $\uu$ is \textit{eventually periodic} if there exist words $x \in \A^*$ and $v \in \A^* \setminus \{\varepsilon\}$ such that $\uu$ can be written as $\uu = xvvv \cdots = xv^\omega$. If $\uu$ is not eventually periodic, $\uu$ is called \textit{aperiodic}.

Consider a factor $v$ of a recurrent sequence $\uu = u_0 u_1 u_2 \cdots$. Let $j < \ell$ be two consecutive occurrences of $v$ in $\uu$. Then the word $u=u_j u_{j+1} \cdots u_{\ell-1}$ is a \textit{return word} to $v$ in $\uu$ and $uv$ is a \textit{complete return word} to $v$ in $\uu$.

A word $p$ is a \textit{palindrome} if $p$ is equal to its mirror image, i.e., $p=\overline{p}$.
If $p$ is a~palindromic factor of a sequence $\uu$ over $\mathcal A$, then the set of palindromic extensions of $p$ is $\mathrm{Pext}(p)=\{ipi \in {\mathcal L}(\uu)\ : \ i \in {\mathcal A}\}$. Every word $u$ contains at most $|u|+1$ palindromic factors (including the empty word). If $u$ contains $|u|+1$ palindromic factors, $u$ is said to be {\em rich}. A~sequence $\uu$ is rich if each of its factors is rich. 

We use the following equivalent characterizations of richness~\cite{GlJuWi2009, BaPeSt2010}. 
\begin{theorem}\label{thm:richness_eqv}
Let $\uu$ be a sequence with language closed under reversal. The following statements are equivalent.
\begin{enumerate}
\item $\uu$ is rich;
\item every complete return word to a palindromic factor of $\uu$ is a palindrome;
\item every bispecial factor $w$ of $\uu$ satisfies
\begin{itemize}
\item ${\mathrm b}(w)=0$ if $w$ is not a palindrome;
\item ${\mathrm b}(w)=\# {\mathrm Pext}(w)-1$ if $w$ is a palindrome.
\end{itemize}
\end{enumerate}
\end{theorem}

A \textit{morphism} is a map $\psi: \A^* \to \B^*$ such that $\psi(uv) = \psi(u)\psi(v)$ for all words $u, v \in \A^*$.
The morphism $\psi$ can be naturally extended to a sequence $\uu=u_0 u_1 u_2\cdots$ over $\A$ by setting
$\psi(\uu) = \psi(u_0) \psi(u_1) \psi(u_2) \cdots\,$.
Consider a factor $w$ of $\psi({\uu})$. We say that $(w_1, w_2)$ is a \emph{synchronization point} of $w$ with respect to $\psi$ if $w=w_1w_2$ and for all $p,s \in {\mathcal L}(\psi({\uu}))$ and $v \in {\mathcal L}({\uu})$ such that $\psi(v)=pws$ there exists a factorization $v=v_1v_2$ of $v$ with $\psi(v_1)=pw_1$ and $\psi(v_2)=w_2s$.

A \textit{fixed point} of a morphism $\psi:  \A^* \to  \A^*$ is a sequence $\uu$ such that $\psi(\uu) = \uu$.
We associate to a morphism $\psi: \A^* \to  \A^*$ the \textit{incidence matrix} $M_\psi$ defined for each $i,j \in \mathcal A$ as $(M_\psi)_{ij}=|\psi(j)|_{i}$. 
A morphism $\psi$ is \textit{primitive} if the matrix $M_\psi$ is primitive, i.e., there exists $k\in \mathbb N$ such that $M_\psi^k$ is a positive matrix.   


 \medskip 
Let ${\uu}$ be a sequence over $\mathcal A$. Then the \textit{uniform frequency} of the letter ${i}\in \mathcal A$ is equal to $f_i$ if for any sequence $(v_{n})$ of factors of ${\uu}$ with increasing lengths 
$$f_i=\lim_{n\to \infty}\frac{|v_{n}|_{i}}{|v_{n}|}\,.$$
It is known that fixed points of primitive morphisms have uniform letter frequencies~\cite{Quef87}.
In order to compute the (asymptotic) critical exponent of sequences, we use the following theorems.~\footnote{Foster et al.~\cite{FoSaVa2025} have recently introduced a~modified version of Theorem~\ref{thm:E*_morphic_image}.}
\begin{theorem}[\cite{DolceDP2023}, Theorem 3]\label{thm:FormulaForE}
Let $\uu$ be a uniformly recurrent aperiodic sequence.
Let $(w_n)_{n\in\N}$ be the sequence of all bispecial factors in $\uu$ ordered by length.
For every $n \in \N$, let $r_n$ be the shortest return word to the bispecial factor $w_n$ in $\uu$.
Then
$$
E(\uu) = 1 + \sup\left\{\frac{|w_n|}{|r_n|}\ : \ n\in \mathbb N \right\}\,.
$$
\end{theorem}

\begin{theorem}[\cite{Ochem}, Theorem 9]\label{thm:E*_morphic_image} Let ${\bf v}$ be a sequence over an alphabet $\mathcal A$ such that the uniform letter frequencies in ${\bf v}$ exist. Let $\psi:{\mathcal A}^* \to {\mathcal B}^*$ be an injective morphism. Assume there exists $L \in \mathbb N$ such that every factor $x$ of $\psi({\bf v})$, $|x|\geq L$, has a~synchronization point. 
Then $E^*({{\bf v}})=E^*(\psi({{\bf v}}))$.
\end{theorem}

The focus of our attention in this article is the fixed point $\uu_d = \uu$ of the morphism $\varphi_{d} = \varphi$ defined in \eqref{eq:morphismFi}.  Even though we focus only on odd indices $d$ later, the properties derived in~\cite{DvKlPe2025} and listed below apply to all indices $d \geq 3$. 
\begin{enumerate}
\item The morphism $\varphi$ is primitive and injective.
\item If $w \in \A^*$ is a palindrome, then $\varphi(w)0$ is a palindrome. Hence, $\uu$ contains infinitely many palindromic factors.
\item The sequence $\uu$ is uniformly recurrent. 
\item The language of $\uu$ is closed under reversal.
\item The sequence $\uu$ is rich.
\item Every factor of $\uu$ has exactly $d$ return words.
\item The asymptotic critical exponent of $\uu$ satisfies 
\begin{equation}\label{eq:Lambda}
E^*(\uu) = 1+\frac{1}{3-\Lambda}\,, \ \text{where \ $\Lambda \in (2,3)$ \ is zero of the polynomial $t^{d-1}(t-2)^2-1$}.
\end{equation}
\end{enumerate}

\section{Palindromic bispecial factors in $\uu_{2d+1}$ and their return words}\label{sec:palindrome}
As follows from Theorem \ref{thm:FormulaForE}, the key objects necessary to determine the critical exponent of a sequence are bispecial factors and their return words.  An efficient method for describing bispecial factors of a~fixed point of a very general morphism was provided and developped in~\cite{Klouda2012, KlSt2015, GOKlSt2025}. This method was applied  in~\cite{DvKlPe2025}. Here we 
 summarize properties of bispecial factors of $\uu_{2d+1}$. We  draw from Proposition 15 and Proposition 19 in~\cite{DvKlPe2025} and use the notation
\begin{equation}\label{eq:viceVidlickoveBS} 
    F_0=\varepsilon \quad \text{and} \quad \ F_k=\varphi(F_{k-1})0\,
    \quad \text{for}\quad  k=1,2,\ldots, 2d\,.
\end{equation}
Note that every $F_k$ is a palindrome and 
$F_k = \varphi^{k-1}(0)\varphi^{k-2}(0)\cdots \varphi(0)0$.
In the sequel, $M$ denotes the incidence matrix of the morphism $\varphi$ defined in~\eqref{def:morphism_matrix}. We distinguish three types of palindromic bispecial factors:

\begin{description}

\item[Type 0] $F_k$ is a bispecial factor for every $k \in \{0,1,\ldots, 2d-1\}$. 

\item[Type I] An infinite series of palindromic bispecial factors $(w_n)_{n \in \mathbb N}$, where  $w_0 = (2d)$ has the shortest return word satisfying $\vec{r}_0 =\vec w_0=\vec{e}_{2d}$.

\item[Type II] An infinite series of palindromic bispecial factors $(w_n)_{n \in \mathbb N}$, where $w_0 = (2d)F_{2d}(2d)$ has the shortest return word satisfying $\vec{r}_0 =M^{2d}\vec{e}_0$. 
\end{description}

For the Parikh vectors of palindromic bispecial factors and their shortest return words in both infinite series $(w_n)$ holds 
\begin{equation}\label{eq:PalindromicBS} \vec{r}_{n} = M \vec{r}_{n-1} \qquad \text{and}\qquad   \vec{w}_{n} = M \vec{w}_{n-1}\ +  \ \left\{\begin{array}{ll}
\vec{e}_0+2\vec{e}_{2d}& \text{ if }\ n=0 \!\mod(2d); \\
\vec{e}_0 &\text{  otherwise}\,.
\end{array} \right.
\end{equation}
\begin{remark} Let us stress one important fact we deduced in \cite{DvKlPe2025}: if $w_n$ is a bispecial factor of Type I or II and $r_n$ the shortest return word to $w_n$, then each  return word $r$ to $w_n$ satisfies  $\vec{r}\geq\vec{r}_n$.  This property guarantees that $|\eta(r)|\geq |\eta(r_n)|$ for every morphism $\eta$.  
    
\end{remark}

The following lemmas help us find an explicit form of the Parikh vectors of palindromic bispecial factors.

\begin{lemma}\label{lem:pomocnaf}  Consider a sequence $(\vec{f}_n)_{n \in \N} $ of vectors from $\mathbb{R}^{2d+1}$ defined recursively:  
 $$\vec{f}_0 = \vec{0}\ \    \text{and} \ \ \vec{f}_n = 
M\,\vec{f}_{n-1} +\ \left\{\begin{array}{ll}
\vec{e}_0+2\vec{e}_{2d}& \text{ if }\ n=0 \!\mod(2d); \\
\vec{e}_0 &\text{  otherwise}\,.
\end{array} \right. $$
Then  \   
$
\vec{f}_n =A\,M^n \vec{e}_{2d} - B_i \vec{e}_0,$  \ where\\

 $i\in \{0,1,\ldots, 2d-1\}$ satisfies $i=n \mod (2d)$;

$A= (M-2I)(3I -M)^{-1}$ \  and 

 $B_i = (3I-M)^{-1} + 2 (M-2I)(M-I)^{-1}(M^i-I)(3I-M)^{-1}$. 
   
\end{lemma}

\begin{proof} 
Let $n \in \N$ be of the form $n =2dN + i$, where $i\in\{0,1,\ldots, 2d-1\}$  and $N \in \N$.
The recurrence relation implies 
$$ 
\vec{f}_n= \sum_{j=0}^{n-1}M^j\vec{e}_0  +  2M^i\sum_{j=0}^{N-1}M^{2dj}\vec{e}_{2d}. $$
Our goal is to express $\vec{f}_n $ in the form  $A\,M^n\,\vec{e}_{2d} - B_i\vec{e}_0,$ 
where the matrix $B_i$ depends only on $i$, not on $N$,  and the matrix $A$ depends neither on $N$
nor on $i$. To reach this goal, we make use of the following equalities:
\begin{itemize}
   \item   $\vec{e}_0 = (M-2I)\vec{e}_{2d}$;
    \item $\sum_{j=0}^{n-1} M^j = (M-I)^{-1}(M^n - I) $;
    \item $\sum_{j=0}^{N-1} M^{2dj+i} = (M^{2d}-I)^{-1}(M^{n} -M^i) $.
 \end{itemize}
Consequently, it suffices to put $$A = (M-I)^{-1}(M-2I) + 2(M^{2d}-I)^{-1}\quad \text{and} \quad  B_i=(M-I)^{-1} + 2(M^{2d}-I)^{-1}(M-2I)^{-1}M^i\,.$$
 By Hamilton-Cayley theorem  $M^{2d}(M-2I)^2 = I$, hence
    
    \medskip
    \centerline{$(M^{2d} -I)(M-2I)^2 = I-(M-2I)^2 = (M-I)(3I-M)$\,.}
\noindent Using the above relation, we get $A$ and $B_i$ in the declared form.~\footnote{We write $B_i$ in the two-part form $(3I-M)^{-1} + 2 (M-2I)(M-I)^{-1}(M^i-I)(3I-M)^{-1}$ because it is then obvious for $i=0$ that the second part disappears and $B_0=(3I-M)^{-1}$.}

\end{proof}

\begin{lemma}\label{lem:explicitBS} Let $(\vec{f}_n)_{n\in \mathbb N}$ be a sequence of vectors defined in Lemma \ref{lem:pomocnaf}. Then 
\begin{enumerate}

   \item $\vec{f}_k$ is the Parikh vector of the bispecial factor $F_k$ for $k=0,1,\ldots, 2d-1$. Moreover, $\vec{f}_k\leq \vec{r}$ for every return word $r$ to the factor $F_k$.    \item 
        $\vec{f}_n + M^{n}\vec{e}_{2d}$ is the Parikh vector of the bispecial factor $w_n$ of Type I for every $n\in \N$. The Parikh vector of the  shortest return word  $r_n$  to $w_n$ 
is $M^n\vec{e}_{2d}$. 
    \item $\vec{f}_{n+2d}$ is the Parikh vector of the bispecial factor $w_n$ of Type II for every $n\in \N$. The Parikh vector of the shortest return word $r_n$ to $w_n$ is $M^{n+2d}\vec{e}_{0}$.

\end{enumerate}
    
\end{lemma}
\begin{proof} 
The Parikh vectors of the bispecial factors $F_k$ defined by \eqref{eq:viceVidlickoveBS} satisfy the recurrence relation  $\vec{F}_0 =\vec{0}$ and $\vec{F}_{k} = M\vec{F}_{k-1} +\vec{e}_0$ with $\vec{F}_0 = \vec{f_0}$.   Hence  $\vec{F}_k = \vec{f_k}$ for $k = 0,1, \ldots, 2d-1$.   

As $F_0 = \varepsilon$, every letter $j \in\{0,1,\ldots, 2d\}$ is a return word to $F_0$ and has the shortest length, namely 1. Obviously, $\vec{f}_0 = \vec{0}\leq  \vec{e_j}$. 

$F_1 = 0$
 and $\varphi(j)$ for $j = 0,\ldots, 2d $
 are return words to $F_1$.  As $\varphi(2d-1) = 0(2d)$ is a proper prefix of $\varphi(2d)=0(2d)(2d)$, the return word $\varphi(2d)$ can be excluded from our further considerations.  The Parikh vector $M\vec{e_j}$ of the remaining return words $\varphi(j)$ satisfies $M\vec{e_j} = \vec{f_1}+ \vec{e}_{j+1}\geq \vec{f}_1$.  

It follows by induction that the shortest return words to $F_k$ are $\varphi^k(j)$ for $j=0, 1, \ldots, 2d-k$ and the Parikh vector of $\varphi^k(j)$ is $\vec{f}_k +\vec{e}_{j+k}\geq \vec{f}_k$.

 \medskip
 
The form of recurrence relation~\eqref{eq:PalindromicBS} implies for both types of palindromic bispecial factors 
 $$\vec{w}_n = M^n\vec{w}_0 + \vec{f}_n\quad \text{and}\quad \vec{r}_n=M^n\vec{r}_0\,.$$
 For Type I, the initial vectors are $\vec{w}_0 = \vec{r}_0=\vec{e}_{2d}$, for Type II, the initial vectors are $\vec{w}_0 = 2\vec{e}_{2d} + \vec{F}_{2d} =  2\vec{e}_{2d} + M\vec{F}_{2d-1} + \vec{e}_0  =   2\vec{e}_{2d} + M\vec{f}_{2d-1}+ \vec{e}_0  =  \vec{f}_{2d}$ and $\vec{r}_0 =M^{2d}\vec{e}_0$. 
 Using Lemma~\ref{lem:pomocnaf}, we get that $M^n\vec{f}_{2d}+\vec{f}_n=\vec{f}_{n+2d}$.
    
\end{proof}

\section{Non-palindromic bispecial factors in $\uu_{2d+1}$  and their return words}\label{sec:nonpalindrome}

By Proposition~19 in~\cite{DvKlPe2025}, there is an infinite series of non-palindromic bispecial factors  $(w_n)_{n \in \mathbb N}$ associated with each $k \in \{1,2,\ldots, 2d-1\}$, where the initial bispecial factor equals $w_0=F_k(2d)$ and the Parikh vector of the shortest return word $r_0$ to $w_0$ satisfies
$\vec{r_0}\geq M^k\,\vec{e}_{2d}$ \  or \    $\vec{r_0}\geq M^{2d}\,\vec{e}_{0}$. As the language of $\uu_{2d+1}$ is closed under reversal, the mirror image of every $w_n$ is a non-palindromic bispecial factor, too. No other non-palindromic bispecial factors occur in $\uu_{2d+1}$.   

For the Parikh vectors of $w_n$ and their shortest return words $r_n$ in all non-palindromic series, the following recurrence relation holds: 
\begin{equation}\label{eq:non-PalindromicBS} \vec{r}_{n} = M \vec{r}_{n-1} \quad \text{and}\quad   \vec{w}_{n} = M \vec{w}_{n-1}\ +  \ \left\{\begin{array}{ll}
\vec{e}_0+\vec{e}_{2d}& \text{ if }\ n =-k  \text{ or } 0 \!\mod(2d)\,;\\
\vec{e}_0 &\text{  otherwise}\,.
\end{array} \right.
\end{equation}

\begin{lemma}\label{lem:NezahrajiSi} Let $w$ be a non-palindromic bispecial factor and $r$ a return word to $w$ in $\uu_{2d+1}$, then 
$$\vec{r}\geq \vec{w}+\vec{e}_{2d}\qquad \text{or} \qquad \vec{r}\geq \vec{w}+\vec{e}_{2d-1} + \vec{e}_{2d-2}+\cdots +\vec{e}_1+\vec{e}_0.
$$
\end{lemma}

\begin{proof} Let  $n\in \N$ and $k \in \{1,2,\ldots, 2d-1\}$ such that  $w=w_n$ and  $w_n$  belongs to a series of non-palindromic bispecial factors starting with $w_0 = F_k(2d)$. 

Case $n =0$.  By the above description of the shortest return word $r=r_0$ to the starting non-palindromic bispecial factor we know that $\vec{r}\geq M^k \vec{e}_{2d}$ or $\vec{r}\geq M^{2d} \vec{e}_{0}$.  

Note that $ w_0 = F_k(2d)$ is a proper prefix of $\varphi^k(2d)$ and the last letter of $\varphi^k(2d)$ is $(2d)$. Thus $\vec{w}+  \vec{e}_{2d}\leq M^k\,\vec{e}_{2d} = \text{the Parikh vector of } \varphi^k(2d) $.

By Lemma 18, Item 6 in \cite{DvKlPe2025},  $\varphi^{2d}(0) = F_{2d}(2d) = \varphi^{2d-1}(0)\varphi^{2d-2}(0) \cdots \varphi^k(0)F_{k}(2d)$. We deduce that $\text{the Parikh vector of } \varphi^{2d}(0)  \geq  M^{2d-1}\vec{e}_0+\vec w \geq \vec{e}_{2d-1} +\vec{e}_{2d-2}+\cdots+\vec{e}_1+\vec{e}_0 +  \vec{w} $, where the last inequality follows from the fact that $ \varphi^{2d-1}(0)$ contains all letters in $\{0,1,\ldots, 2d-1\}$.

Thus the statement is valid for $n =0$.\\[1mm] 

If for some index $n \in \N, n\geq 1 $,  the inequality $ \vec{r}_{n-1}\geq \vec{w}_{n-1} + \vec{e}_{2d}$ takes place, then   
$\vec{r}_{n}= M\vec{r}_{n-1}\geq M\vec{w}_{n-1} + M\vec{e}_{2d} = M\vec{w}_{n-1} + \vec{e}_0+ 2 \vec{e}_{2d} \geq \vec{w}_{n} + \vec{e}_{2d}$.   
\medskip

If for some index $n \in \N, n\geq 1 $,  the inequality
$\vec{r}_{n-1}\geq \vec{w}_{n-1} + \vec{e}_{2d-1} +\vec{e}_{2d-2}+\cdots +\vec{e}_1+\vec{e}_0$ holds, then 
$$\begin{array}{rcl} 
\vec{r}_n & \geq & M\vec{w}_{n-1} + M\vec{e}_{2d-1} +M\vec{e}_{2d-2} +\cdots +M\vec{e}_1+M\vec{e}_0\\
&=& M\vec{w}_{n-1} + 2d\vec{e}_0+\vec{e}_{2d} +\vec{e}_{2d-1}+\vec{e}_{2d-2}+\cdots+\vec{e}_1\\
&\geq & \vec{w}_{n} +\vec{e}_{2d-1}+\vec{e}_{2d-2}+\cdots+\vec{e}_1+\vec{e}_0\,.
\end{array}$$




\end{proof}

\section{Morphism $\pi$ and richness  of $\pi(\uu_{2d+1})$}
\label{sec:pi}
\begin{definition}\label{def:weight} Let $ \mathcal{A}=\{0,1,\ldots, 2d\}$. Put $$\vec{h}=(h_0,h_1, \ldots, h_{2d}) =(\underbrace{1,1, \ldots,1}_{d-\text{items}}, \,\underbrace{2^0,2^1, \ldots, 2^{d}}_{(d+1)-\text{items}} \,)\,.$$
A morphism $\pi: \mathcal{A}^* \mapsto\mathcal{B}^*$
satisfying for every $i\in \mathcal{A}$  
\begin{enumerate} \item $\pi(i)$ is a palindrome of length $h_i$; 
\item the number of distinct letters occurring in $\pi(i)$ is maximum; 
\item any letter which occurs in  $\pi(i)$  does not occur in $\pi(j)$ for $j\neq i$
\end{enumerate} 
is called a~{\em weighted morphism}. 
\end{definition}

Let us remark that the cardinality of the alphabet $\mathcal B$ equals $d+2^d$.
 \begin{example} Let $d=2$. Then $\vec{h}=(1,1,1,2,4)$ and 
 $\pi(0) = 0$, $\pi(1) = 1$, $\pi(2) = 2$, $\pi(3) = 33$, 
$\pi(4) = 4554$.      
 \end{example}

\begin{proposition}\label{prop:Projekce} Let $\pi$ be a weighted morphism over $ \mathcal{A}=\{0,1,\ldots, 2d\}$ and $\uu = \uu_{2d+1}$.    
\begin{itemize} 
\item A factor $p$ of $\uu$ is a~palindrome if and only if $\pi(p)$ is a~palindrome.  
\item  The sequence $ \pi(\uu)$ is rich.  
\item If $w'$ is a bispecial factor of $ \pi(\uu)$ of length $>1$ and $r'$ a return word to $w'$, then 
$w' = \pi(w)$  and $r' = \pi(r)$, where  $w$ is a bispecial factor of $\uu$ and $r$ is a~return word to $w$ in $\uu$.

 \end{itemize}  
 \end{proposition}
\begin{proof}
\begin{itemize}
\item The first item follows immediately from the form of $\pi$.
\item By Theorem~\ref{thm:richness_eqv}, Item 2, it suffices to check that any complete return word to a~non-empty palindrome in $\pi(\uu)$ is a palindrome.  
\begin{itemize}
\item Firstly, consider $p$ a~letter. Then $p$ is contained in the image $\pi(a)$ of a unique letter $a\in\mathcal A$. Then a~complete return word $z$ to $p$ may take on one of the possible forms: either $z$ is entirely contained in $\pi(a)$ and $z$ is a~palindrome by the definition of $\pi$, or $z$ is contained in the factor $\pi(aya)$, where $aya$ is a complete return word to $a$, consequently $aya$ is a palindrome by richness of $\uu$. Then $z$ is a~palindrome, too -- it is a~central factor of $\pi(aya)$, which is a~palindrome by the first item. 
\item Secondly, consider a~palindrome $p$ of length $>1$ in $\pi(\uu)$, which is contained in $\pi(a)$ for some letter $a \in \mathcal A$. Then by the form of $\pi$, the letter $a$ is uniquely given and $p$ occurs once in $\pi(a)$ -- as a central factor. Then any complete return word $z$ to $p$ is contained in the factor $\pi(aya)$, where $aya$ is a complete return word to $a$ in $\uu$, hence $aya$ is a palindrome by richness of $\uu$. Consequently, $z$ is a palindrome, too.    
\item Thirdly, assume $p$ is a palindrome in $\pi(\uu)$, which is not contained in the letter images by $\pi$. By the definition of $\pi$, it can be uniquely written as $p=x\pi(q)\overline{x}$, where $x$ is a proper suffix of $\pi(a)$ for a letter $a\in \mathcal A$ and $q$ is a factor of $\uu$. The words $x$ and $q$ may be empty, but not at the same time. By the first item, $q$ is a~palindrome. 
If $z$ is a complete return word to $p$ in $\pi(\uu)$ and $q$ is non-empty, then again it can be uniquely written as $z=x\pi(y)\overline{x}$, where $y$ is a complete return word to $q$. Hence, $y$ is a palindrome due to richness of $\uu$. Consequently, $z$ is a palindrome, too, by the first item.
If $q$ is empty, then $p=x\overline{x}$ occurs as the central factor of $\pi(aa)$. By similar arguments as above, any complete return word $z$ to $p$ is a~central factor of $\pi(y)$, where $y$ is a complete return word to $aa$ in $\uu$, hence $y$ is a palindrome and so is $z.$
\end{itemize}
\item If $w'$ is a bispecial factor of $ \pi(\uu)$ of length $>1$, then the form of $\pi$ implies that $w'=\pi(w)$ and $w$ is a bispecial factor in $\uu$. If $r'$ is a return word to $w'$, then $r'w'$ starts in $w'$ and ends in $w'$, therefore $r'w'=\pi(rw)$, where $rw$ contains $w$ as prefix and suffix and nowhere else. Thus, $r$ is a return word to $w$ in $\uu$.
\end{itemize}
    
\end{proof}

\section{Proof of Theorem \ref{thm:hlavni}}\label{sec:proofThm} 
If $w'$ is a~bispecial factor and $r'$ is a~return word to $w'$ in $\pi(\uu_{2d+1})$, then by Proposition~\ref{prop:Projekce}, we obtain for length the following formula $|w'| = \vec{h}\vec{w}$ and $|r'| = \vec{h}\vec{r}$, where $w$ and $r$ are a~bispecial factor and its return word in $\uu_{2d+1}$. 
Given the relation \eqref{eq:Lambda} and Theorem~\ref{thm:FormulaForE},  our goal is to prove that for every bispecial factor $w$ of $\uu_{2d+1}$, where $d\geq 3$, and a return word $r$ to $w$, the following inequality holds:
\begin{equation}\label{eq:nerovnost}
\frac{\vec{h}\, \vec{w}}{ \vec{h}\, \vec{r}} \ \leq \   \frac{1}{3-\Lambda},  \qquad \text{equivalently,}\qquad \vec{h}\, \vec{w}   \ \leq \ \frac{1}{3-\Lambda}  \  \vec{h}\, \vec{r}\,.  
    \end{equation}
The proof of this inequality is divided into three lemmas depending on the type of bispecial factor.     
\begin{description}
\item[ Bispecial factors of Type 0 and non-palindromic bispecial factors\ ] 
\begin{lemma}\label{lem:NonPalNE}  Let $w$ be a bispecial factor in $\uu_{2d+1}$ and $r$ a return word to $w$. If $w$ equals $F_k$ for some $k = 0,1,\ldots, 2d-1$ or $w$ is non-palindromic,  then  the inequality  \eqref{eq:nerovnost} holds true.  
    \end{lemma}
\begin{proof}  Due to Lemma \ref{lem:NezahrajiSi} and Lemma \ref{lem:explicitBS}, Item 1, 
   $\vec{w}\leq \vec{r}$. Hence $\vec{h}\,\vec{w} \leq \vec{h}\, \vec{r}$. Since $1<\frac{1}{3-\Lambda}$, the inequality \eqref{eq:nerovnost} holds. 
\end{proof}
\item[ Bispecial factors of Type I\ ]  
\begin{lemma}\label{lem:Dukaz nerovnostiI} 
The inequality~\eqref{eq:nerovnost} holds for all bispecial factors of Type I and their return words. 
\end{lemma}
\begin{proof}
Using Lemma~\ref{lem:explicitBS}, the inequality \eqref{eq:nerovnost} can be rewritten as $\vec{h}\,\vec{f}_n + \vec{h}\,M^n\,\vec{e}_{2d}  \leq  \frac{1}{3-\Lambda} \,\vec{h}\, M^{n}\vec{e}_{2d}$ for $n \in \mathbb N$. Applying Lemma~\ref{lem:pomocnaf}, we get 
\begin{equation}\label{eq:ProTypI}
   \vec{h}\, \Bigl( A - \tfrac{\Lambda-2}{3-\Lambda}\,I\Bigr)\,M^n \,\vec{e}_{2d}\ \leq \ \vec{h}\, B_i \vec{e}_0  \qquad \text{for } n \in \N. 
\end{equation}
We will show that the left side of the inequality~\eqref{eq:ProTypI} is $<1$, while the right side is $>1$. This will prove the statement. 
\begin{description}
\item[The left side of \eqref{eq:ProTypI}] \quad Using  Lemma \ref{lem:ReseniRekurence}, we have
  \begin{equation}\label{eq:KonkreteTyp1} \vec{h}\Bigl(A - \tfrac{\Lambda - 2}{3-\Lambda}I\Bigr) M^n\,\vec{e}_{2d}  = \sum_{\lambda \in S_{\Lambda}\setminus\{\Lambda\}}  \tfrac{1}{(d+1)\lambda-2d}\Bigl(\tfrac{\lambda - 2}{3-\lambda}-\tfrac{\Lambda - 2}{3-\Lambda}\Bigr)\lambda^{n+d+1}\,,\end{equation}
  where $S_\Lambda$ is the set of algebraic conjugates of $\Lambda$.
  Let us estimate the modulus of the left side. Each summand may be estimated in the same way. We use the fact that for $\lambda \neq \Lambda $ we have $|\lambda| < 1$,  and moreover
  
  $|(d+1)\lambda -2d| \geq 2d- |(d+1)\lambda| \geq d-1$;

  $\Bigl|\tfrac{(\lambda -2)\lambda^d}{3-\lambda}\Bigr| = \Bigl|\tfrac{1}{3-\lambda}\Bigr|\leq \tfrac12$;

  $\frac{\Lambda - 2}{3-\Lambda} = \tfrac{1}{\Lambda^d-1}\leq \tfrac{1}{2^d-1}$. 

  Since the sum has $d$ summands, the modulus of the left side is $\leq \frac{d}{d-1}(\frac12 + \tfrac{1}{2^d-1})$. This estimate decreases with increasing value $d$. As we consider $d\geq 3$, we get
$\frac{d}{d-1}(\frac12 + \tfrac{1}{2^d-1}) \leq \tfrac{27}{28}< 1$.

\item[The right side of \eqref{eq:ProTypI}]  \quad   It is easy to verify that  the first  column of the matrix $(3I-M)^{-1}$, i.e., the vector  $(3I-M)^{-1}\vec{e}_0$, equals 
$
\tfrac{2}{3^{2d}-1}\bigl(3^{2d-1},3^{2d-2}, \ldots, 3^1, 3^0,1  \bigr)^T$,  and $\vec{h}\, (M-2I) = (\underbrace{0,0, \ldots, 0}_{d-\text{times}}, \underbrace{1,1,\ldots, 1}_{(d+1)-\text{times}} )$.

Let us write $B$ in the form $B = B' + B"$, where $$B'= (3I-M)^{-1}\qquad \text{and} \qquad  B"=  2 (M-2I)(M-I)^{-1}(M^i-I)(3I-M)^{-1} .$$

A straightforward calculation gives 
$\vec{h} \,B'\,\vec{e}_0 = \vec{h}\,(3I-M)^{-1}\vec{e}_0 = \tfrac{3^d}{3^d-1}\,.
$

The number $\vec{h} \,B"\,\vec{e}_0$ is non-negative because it is product of a~non-negative vector $2\vec{h}(M-2I)$, of a~non-negative matrix $(M-I)^{-1} (M^i-I) = M^{i-1} + M^{i-2} + \cdots+ I $ and of a~non-negative vector $(3I - M)^{-1}\vec{e}_0$.

We can conclude that the right side of \eqref{eq:ProTypI} is  $\geq \tfrac{3^d}{3^d-1}>1$. 
\end{description}

\end{proof} 
\item[ Bispecial factors of Type II\ ]  
\begin{lemma}\label{lem:Dukaz nerovnostiII}  The inequality~\eqref{eq:nerovnost} holds for all bispecial factors of Type II and their return words.      
\end{lemma}
\begin{proof} 
By Lemma \ref{lem:explicitBS}, the inequality \eqref{eq:nerovnost} can be rewritten as $\vec{h}\vec{f}_{n+2d} \leq  \frac{1}{3-\Lambda} \,\vec{h}\, M^{n+2d}\vec{e}_{0}$ for $n \in \N$. Equivalently, $\vec{h}\vec{f}_{n} \leq  \frac{1}{3-\Lambda} \,\vec{h}\, M^{n}\vec{e}_{0}$ for $n \in \N$, $n\geq 2d$. We again use the explicit form of $\vec{f}_n$ from Lemma \ref{lem:pomocnaf} and the equality $\vec{e}_0 = (M-2I)\vec{e}_{2d}$ to obtain 

\begin{equation}\label{eq:ProTypII}
   \vec{h}\, \Bigl( A - \tfrac{1}{3-\Lambda}\,(M-2I)\Bigr)\,M^n \,\vec{e}_{2d}\ \leq \ \vec{h}\, B_i \vec{e}_0 \qquad \text{for } n \in \N, n\geq 2d.
\end{equation}

Using  Lemma \ref{lem:ReseniRekurence}, we obtain
  $$  \vec{h}\, \Bigl( A - \tfrac{1}{3-\Lambda}\,(M-2I)\Bigr)\,M^n \,\vec{e}_{2d} = \sum_{\lambda \in S_{\Lambda}\setminus\{\Lambda\}}  \tfrac{1}{(d+1)\lambda-2d}\Bigl(\tfrac{\lambda - 2}{3-\lambda}-\tfrac{\lambda - 2}{3-\Lambda}\Bigr)\lambda^{n+d+1}\,. $$

Let us rewrite the expression
\begin{equation}\label{eq:skorostejne}
\Bigl(\tfrac{\lambda - 2}{3-\lambda}-\tfrac{\lambda - 2}{3-\Lambda}\Bigr)\lambda^{n+d+1}= (\lambda - 2) \,\tfrac{\lambda -\Lambda}{(3-\lambda)(3-\Lambda)}\,\lambda^{n+d+1}= \tfrac{\lambda -\Lambda}{(3-\lambda)(3-\Lambda)}\,\lambda^{n+1}\,.\end{equation}

Let us emphasize that we used the equality $(\lambda-2)\lambda^d = 1$ for each algebraic conjugate of $\Lambda$, i.e., for $\lambda \in S_\Lambda$. 

Let us compare the expression in~\eqref{eq:skorostejne} with a~rewritten expression from~\eqref{eq:KonkreteTyp1}.

\begin{equation}\label{eq:skorostejne2}\Bigl(\tfrac{\lambda - 2}{3-\lambda}-\tfrac{\Lambda - 2}{3-\Lambda}\Bigr)\lambda^{n+d+1} = \tfrac{\lambda -\Lambda}{(3-\lambda)(3-\Lambda)}\,\lambda^{n+1+d}\,.\end{equation}
\medskip

\noindent Therefore from the validity of inequality~\eqref{eq:ProTypI} for all $n \in \N$, the inequality~\eqref{eq:ProTypII} follows for all $n \geq d$.    
\end{proof}
\end{description}
\begin{proof}[Proof of Theorem~\ref{thm:hlavni}]
On the one hand, Lemmas~\ref{lem:NonPalNE}, \ref{lem:Dukaz nerovnostiI}, and~\ref{lem:Dukaz nerovnostiII} show that for every bispecial factor $w$ in $\uu_{2d+1}$ and any of its return words $r$, the inequality~\eqref{eq:nerovnost} holds, i.e., $E(\pi(\uu_{2d+1}))\leq 1+\frac{1}{3-\Lambda}=E^*(\pi(\uu_{2d+1}))$. 

On the other hand, by Theorem~\ref{thm:E*_morphic_image}, we have $E^*(\pi(\uu_{2d+1}))=E^*(\uu_{2d+1})=1+\frac{1}{3-\Lambda}$. The assumptions of Theorem~\ref{thm:E*_morphic_image} are met since $\uu_{2d+1}$ is a fixed point of a primitive morphism and consequently has uniform letter frequencies and all factors of $\pi(\uu_{2d+1})$ of length $\geq 2^d$ have a synchronization point. Since $E(\pi(\uu_{2d+1}))\geq E^*(\pi(\uu_{2d+1}))$, the proof of Theorem~\ref{thm:hlavni} is complete. 
\end{proof}
\section{Comments on images of $\uu_{2d+1}$ under morphisms to smaller alphabets}\label{sec:comments}

The morphism  $\pi$ we work with in this paper maps the sequence $\uu_{2d+1}$ over the alphabet of size $2d+1$ onto an alphabet of size $d+2^d$.
However, as follows from the papers~\cite{CuMoRa2020, CuMoPe2024, DvPe2025} for $d=2$  and $d=3$, there exists a~morphism $\eta$ onto the alphabet of size $d$ such that $E^*(\uu_{2d-1}) = E(\eta(\uu_{2d-1}))$.  To prove at least the easier part of Conjecture \ref{con:opet}, namely $RT(\mathcal{C}_d) \leq E^*(\uu_{2d-1})$, it would be sufficient to find such a morphism for each $d$. The following propositions could provide a hint on how to find such morphisms.

\begin{proposition} Let  $d \in \N$ and $\eta:\{0,1,\ldots, 2d\}^* \mapsto \mathcal{B}^*$ be a morphism  satisfying 
\begin{enumerate}
    \item there exists a finite set $\mathcal{F} \subset \mathcal{B}^*$ such that every sufficiently long bispecial factor in $\eta(\uu_{2d+1})$ and its return word have the form $p\eta(w)s$  and $p\eta(r)p^{-1}$, respectively,   with $p,s \in \mathcal{F}$ and $w$ a~bispecial factor of $\uu_{2d+1}$ and $r$ a~return word to $w$ in $\uu_{2d+1}$.  
    \item $E(\eta(\uu_{2d+1}))= E^*(\uu_{2d+1})$. 
\end{enumerate}

Then the integer row vector $\bigl(|\eta(0)|, |\eta(1)|, \ldots, |\eta(2d)|\bigr)$ is orthogonal to an eigenvector of the incidence matrix $M$ of the morphism $\varphi_{2d+1}$ corresponding to the eigenvalue $\beta\in (1,2)$.
\end{proposition}
\begin{proof}
Let $\vec{h}$ denote the row vector $\bigl(|\eta(0)|, |\eta(1)|, \ldots, |\eta(2d)|\bigr)$.

By Item 1 of the assumptions,  the lengths of a~sufficiently long bispecial factor  in $\eta(\uu_{2d+1})$ and any its return word is $\vec{h}\,\vec{w} + |s|+ |p|$  and $\vec{h}\,\vec{r}$, respectively, where $s,  p \in \mathcal{F}$.  

By Item 2 of the assumptions, $\vec{h}\,\vec{w} + |s|+ |p| \leq \frac{1}{3-\Lambda}\, \vec{h}\,\vec{r} $.   Let us specify  this inequality for the sufficiently long palindromic bispecial factors of $\uu_{2d+1}$.

For a sufficiently long  bispecial factor $w_n$ of Type I, the inequality reads (cf. \eqref{eq:ProTypI})
\begin{equation}\label{eq:ProTypIObecne}
   \vec{h}\, \Bigl( A - \tfrac{\Lambda-2}{3-\Lambda}\,I\Bigr)\,M^n \,\vec{e}_{2d} + |s_n| + |p_n|\ \leq \ \vec{h}\, B_i \vec{e}_0,  
\end{equation}

and for a bispecial factor $w_n$ of Type II (cf. \eqref{eq:ProTypII})
\begin{equation}\label{eq:ProTypIIObecne}
  \vec{h}\, \Bigl( A - \tfrac{1}{3-\Lambda}\,(M-2I)\Bigr)\,M^n \,\vec{e}_{2d} + |s_n| + |p_n|\ \leq \ \vec{h}\, B_i \vec{e}_0.  
\end{equation}

Let $S_M$ be the set of all eigenvalues of the incidence matrix $M$ of the morphism $\varphi_{2d+1}$.  Recall (see \eqref{eq:rovnice} in Appendix) that the characteristic polynomial of $M$ has just two eigenvalues outside the unit circle, $\Lambda \in (2,3)$ and $\beta \in (1,2)$. 

By Lemma \ref{lem:obecnaM},  $\vec{h}\, M^n \vec{e}_{2d}$ equals 
$ \sum_{\lambda \in  S_M } d_\lambda \lambda^n$ for some constants  $d_\lambda$'s.
 Moreover, the constant $d_\lambda$  is zero if and only if $\vec{h}$ is orthogonal to an eigenvector corresponding to the eigenvalue $\lambda \in S_M$.

Let us stress that $\vec{h}\, M^n \vec{e}_{2d}$ is a sum of $2d+1$ expressions of the type $d_\lambda \lambda^n$.  All but two of them (corresponding to $\lambda = \Lambda$ and $\lambda=\beta $) are bounded sequences.

As $A=(M-2I)(3I-M)^{-1}$ and $s_n, p_n$ belong to a finite set $\mathcal{F}\subset \mathcal{B}^*$,  the inequalities   \eqref{eq:ProTypIObecne} and  \eqref{eq:ProTypIIObecne} have the form 
$$
d_\beta \, (\tfrac{\beta -2}{3-\beta} - \tfrac{\Lambda -2}{3-\Lambda}) \, \beta^n \  \leq  \ a_n\,, \quad \text{where $(a_n)$ \ is a bounded sequence}
$$ and 
 $$
d_\beta \, (\tfrac{\beta -2}{3-\beta} - \tfrac{\beta -2}{3-\Lambda}) \, \beta^n \  \leq  \ b_n\,, \quad \text{where $(b_n)$ \ is a bounded sequence}.
$$ 

For the last step of the proof,  we take advantage of the facts that $\lim_{n \to \infty}\beta^n  = +\infty$  and $\beta^d(\beta -2) = -1$, see \eqref{eq:rovnice}.  Let us adjust the expressions on the left sides of the inequalities. 
$$\tfrac{\beta -2}{3-\beta} - \tfrac{\Lambda -2}{3-\Lambda}=\tfrac{\beta -\Lambda}{(3-\beta)(3-\Lambda)} \qquad \text{and}\qquad  
\tfrac{\beta -2}{3-\beta} - \tfrac{\beta -2}{3-\Lambda}=(\beta -2)\tfrac{\beta -\Lambda}{(3-\beta)(3-\Lambda)}= - \tfrac{1}{\beta^d} \tfrac{\beta -\Lambda}{(3-\beta)(3-\Lambda)}. $$
Obviously, if the coefficient $d_\beta$ is non-zero, the left-hand sides of the previous inequalities have opposite infinite limits. This means that at least one inequality does not hold, which is impossible. Hence,  $d_\beta = 0$, equivalently,  the vector $\vec{h}$ is orthogonal to an eigenvector  of $M$ corresponding to the eigenvalue $\beta$.  
\end{proof}

Last but not least, let us explain that it seems reasonable to search for a suitable morphism $\eta$ in the class $P_{\mathrm {ret}}$ defined in~\cite{BaPeSt2011}.
\begin{definition}\label{def:Pret}
A morphism $\eta: {\mathcal A}^* \to {\mathcal B}^*$ is of {\em class} $P_{\mathrm {ret}}$ if $\eta$ is injective and there exists a palindrome $p \in {\mathcal B}^*$ such that 
for all $i \in \mathcal A$ 
$$\eta(i)=p q_i\,,$$
where $q_i$ is a palindrome and $\eta(i)p$ is a complete return word to $p$.
\end{definition}

Durand~\cite{Dur98} introduced the notion of derived sequence: Let ${\bf v}$ be a~uniformly recurrent sequence over $\mathcal{A}$ and let $x$ be a prefix of ${\bf v}$. There exist only finitely many return words to $x$ in ${\bf v}$,  denote them $r_0, r_1, \ldots, r_{k-1}$.  The sequence ${\bf v}$  can be written as a~concatenation of the return words, say ${\bf v} = r_{i_0}r_{i_1}r_{i_2}\cdots$. The sequence $i_0i_1i_2 \ldots$ over the alphabet $\{0,1,\ldots, k-1\}$ is called the {\em derived sequence} of ${\bf v}$  to the prefix $x$.    

On the one hand, as proven in~\cite{BaPeSt2011}, if $p$ is a palindromic factor of a~uniformly recurrent rich sequence~${\bf v}$, then a suffix ${\bf v'}$ of ${\bf v}$ is an image of a~rich sequence $\uu$ under a~morphism from the class $P_{\mathrm {ret}}$ and $\uu$ is the derived sequence of ${\bf v'}$ to the prefix $p$.  
On the other hand, the image of a~rich sequence $\uu$ by a morphism $\eta \in P_{\mathrm {ret}}$ may not be rich. 

\medskip

For morphisms $\eta$ of class $P_{\mathrm {ret}}$, to show richness of $\eta(\uu_d)$, we do not have to check the bilateral orders of all bispecial factors, as required by Theorem~\ref{thm:richness_eqv}, Item 3. It suffices to check the bilateral orders of the bispecial factors that are contained in the letter images by $\pi$ followed by the common palindrome $p$. See the next proposition. 

\begin{proposition} 
Let $\eta: {\{0,1,\dots, d-1\}}^* \to {\mathcal B}^*$ be a morphism of class $P_{\mathrm {ret}}$ from Definition~\ref{def:Pret}. 
If every bispecial factor $w$ in $\eta(\uu_d)$, which does not contain the palindrome $p$, satisfies
\begin{itemize}
\item ${\mathrm b}(w)=0$ if $w$ is not a palindrome;
\item ${\mathrm b}(w)=\# {\mathrm Pext}(w)-1$ if $w$ is a palindrome,
\end{itemize}
then $\eta(\uu_d)$ is rich.
\end{proposition}
\begin{proof}
In order to prove the statement, it suffices to check the condition on bilateral orders from Theorem~\ref{thm:richness_eqv}, Item 3, for bispecial factors of $\eta(\uu_d)$ containing the palindrome $p$. Every bispecial factor $v$ in $\eta(\uu_d)$ containing $p$ arises from a bispecial factor $z$ in $\uu_d$ in the following way:
If $v=xp\cdots py$, where we highlight the first and the last occurrence of $p$ in $v$, then $v=x\eta(w)py$, where $w$ is a bispecial factor (possibly empty) in $\uu_d$ and $x={\mathrm lcs}\{\eta(a), \eta(b)\}$ and $py={\mathrm lcp}\{\eta(c)p, \eta(d)p\}$ for some $a,b$ left extensions of $w$ and $c,d$ right extensions of $w$. In the sequel, we use the knowledge of both-sided extensions of bispecial factors $w$ in $\uu_d$ from~\cite{DvKlPe2025}.
\begin{enumerate}
\item Assume $w$ has two left and two right extensions. 
\begin{itemize}
\item If $w$ is a palindrome, then ${\mathrm {Bext}}(w)=\{aw(d-1), (d-1)wa, (d-1)w(d-1)\}$ or ${\mathrm {Bext}}(w)=\{aw(d-1), (d-1)wa, awa\}$. Consequently, 
$v=x\eta(w)p\overline{x}$, where $x={\mathrm {lcs}}\{\eta(a), \eta(d-1)\}$. Since $\eta \in P_{\mathrm {ret}}$, the factor $v$ is a~palindrome and ${\mathrm b}(v)=0=\#{\mathrm {Pext}}(v)-1$.
\item If $w$ is not a~palindrome, then $\eta(w)p$ is not a palindrome and $v=x\eta(w)py$ is not a~palindrome, either.  Since ${\mathrm b}(w)=0$, clearly ${\mathrm b}(v)=0$.
\end{itemize}
\item Assume $w$ has more left or right extensions. Then $w=F_i=\varphi^{i-1}(0)\varphi^{i-2}(0)\cdots \varphi(0)0$ for some $0\leq i\leq d-3$ (for $i=0$ we set $w=\varepsilon$) and 
$${\mathrm {Bext}}(w)=\{iwj, jwi, (d-1)w(d-1) \ : \ i<j \leq d-1\}\,.$$
\begin{enumerate}
\item Consider $j\leq d-1$, then $iwj$ and $jwi$ give rise to the palindromic bispecial factor $v=x\eta(w)p\overline{x}$, where $x={\mathrm {lcs}}\{\eta(i), \eta(j)\}$. If $ax$ is a suffix of $\eta(i)$ and $bx$ is a suffix of $\eta(j)$, then $iwj, jwi$, resp. $(d-1)w(d-1)$ give rise to the both-sided extensions of $v$: $avb, bva$, resp. $bvb$ in case $j=d-1$. 

If $cx$ is a suffix of $\eta(k)$ for some $k\not =i,j$, then $iwk$ and $kwi$ give rise to the both-sided extensions $avc, cva$ of $v$. If $k=d-1$, then $cvc$ is also a both-sided extension. If $c=b$, the extensions $avc, cva$ are not new. If $c=a$, we get the new extension $awa$. If $c\not =a,b$, we have new extensions $avc, cva$, eventually $cvc$ for $k=d-1$.

To summarize, the set ${\mathrm {Bext}}(v)$ of both-sided extensions of $v$ is the union of the following sets (some of them may be empty):
\begin{itemize}
\item $\{avb, bva\}$;
\item $\bigcup_{i<k\leq d-1, k\not =j} \{c_k va, avc_k \ : \ c_kx \ \text{is suffix of $\eta(k)$ and $c_k\not=a,b$}\}$;
\item $\{ava \ : \ ax \ \text{is suffix of $\eta(k)$ for some $i<k\leq d-1$}\}$;
\item $\{c_{d-1}v c_{d-1}\ : \ c_{d-1}x \ \text{is suffix of $\eta(d-1)$ and $c_{d-1}\not =a$}\}\,. $
\end{itemize}
It is now straightforward to check that ${\mathrm b}(v)=\#{\mathrm {Pext}}(v)-1$.
\item Consider $j< d-1$, then $iwj$ and $(d-1)w(d-1)$ give rise to the bispecial factor $v=x\eta(w)py$, where $x={\mathrm {lcs}}\{\eta(i), \eta(d-1)\}$ and $py={\mathrm {lcp}}\{\eta(j)p, \eta(d-1)p\}$. If $y=\overline{x}$, we come back to the case (a). Assume thus $y\not =\overline{x}$, i.e., $v$ is not a~palindrome. Let us treat the case where $|y|>|x|$, i.e., $y$ has the prefix $\overline{x}$. The opposite case may be treated analogously.

If $ax$ is a suffix of $\eta(i)$ and $bx$ is a suffix of $\eta(d-1)$ and $pyc$ is a prefix of $\eta(j)p$ and $pyd$ is a prefix of $\eta(d-1)p$, then $iwj, (d-1)w(d-1), iw(d-1)$ give rise to the both-sided extensions of $v$: $avc, bvd, avd$.  

If for some letter $f\not =c,d$ the word $pyf$ is a prefix of $\eta(\ell)p$ for some $\ell \not= j$, then $iw\ell$ gives rise to the both-sided extension $avf$ of $v$. 

It follows that ${\mathrm b}(v)=0$.
\item Consider $j<d-1$, then $jwi$ and $(d-1)w(d-1)$ give rise to the bispecial factor $\overline{v}$, where $v$ was defined in the previous item. By closedness of $\eta(\uu_d)$ under reversal it holds ${\mathrm b}(\overline v)=0$, too.
\item Consider $j, k\leq d-1, j\not =k$, then $iwj$ and $kwi$ give rise to the bispecial factor $v=x\eta(w)py$, where $x={\mathrm {lcs}}\{\eta(i), \eta(k)\}$ and $py={\mathrm {lcp}}\{\eta(i)p, \eta(j)p\}$. If $y=\overline{x}$, we come back to the case (a). Assume thus $y\not =\overline{x}$, i.e., $v$ is not a~palindrome. Let us treat the case where $|y|>|x|$, i.e., $y$ has the prefix $\overline{x}$. The opposite case may be treated analogously.

If $ax$ is a suffix of $\eta(i)$ and $bx$ is a suffix of $\eta(k)$ and $pyc$ is a prefix of $\eta(j)p$ and $pyd$ is a prefix of $\eta(i)p$, then $iwj, kwi, jwi$ give rise to the both-sided extensions of $v$: $avc, bvd, avd$. 

If for some letter $e\not =a,b$  the word $ex$ is a suffix of $\eta(\ell)$ for some $\ell>i$, then $\ell wi$ give rise to the both-sided extensions $evd$ of $v$.

If for some letter $f\not =c,d$ the word $pyf$ is a prefix of $\eta(\ell)p$ for some $\ell >i$, then $iw\ell$ gives rise to the both-sided extension $avf$ of $v$. By the way, the word $\ell w i$ provides the both-sided extension $avd$ of $v$, which is not new.

It follows that ${\mathrm b}(v)=0$.
\end{enumerate}

\end{enumerate}
\end{proof}

Based on the known repetition threshold of binary and ternary rich sequences and on some observations for quaternary rich sequences, we expect that suitable candidates for rich sequences with minimum critical exponent will show a large degree of symmetry~\cite{DvPe2025}.
\section{Acknowledgment} We are indebted to our colleague Zuzana Masáková for consultation on some issues related to the irreducibility of polynomials.

\section{Appendix}

First, let us recall a simple property of square matrices. 

\begin{lemma}\label{lem:obecnaM}  Let $M \in \mathbb{R}^{k\times k}$ have $k$ distinct eigenvalues. For every eigenvalue $\lambda$,  a row vector $\vec{u}_\lambda$ and a column vector $\vec{v}_\lambda$ denote a left and  a right eigenvectors of the matrix $M$ corresponding to $\lambda$. Then for every $k$-dimensional row vector $\vec{h}$,    every $k$-dimensional column vector $\vec{g}$ and every $n \in \N$ 
$$
\vec{h}\,M^n \vec{g} = \sum_{\lambda\in S} a_\lambda \, \lambda^n, $$
where 
$S$ denotes the set of eigenvalues of $M$  and  $a_\lambda =  \frac{1}{\vec{u}_\lambda\,\vec{v}_\lambda}\, \bigl(\vec{h}\,\vec{v}_\lambda\bigr) \, \bigl(\vec{u}_\lambda\, \vec{g}\bigr)$ .

\end{lemma}
\begin{proof} Let $\lambda, \mu \in S$. As  $\lambda\, \vec{u}_\lambda \, \vec{v}_\mu = \vec{u}_\lambda M \vec{v}_\mu = \mu \, \vec{u}_\lambda\,\vec{v}_\mu$, we deduce that if $\lambda \neq \mu$, then $\vec{u}_\lambda \, \vec{v}_\mu = 0$. In other words, if columns of a matrix $R \in \mathbb{R}^{k\times k}$ are formed by eigenvectors of $M$ corresponding to distinct eigenvalues, then the rows of the inverse matrix $R^{-1}$ are formed by the row vectors $\vec{u}_\lambda$ normalized by $\frac{1}{\vec{u}_\lambda\,\vec{v}_\lambda}$.   Obviously,  $M = R\,D\,R^{-1}$, where $D$ is a diagonal matrix with eigenvalues $\lambda \in S$ on its diagonal.  Substituting $M^n$ by $ R\,D^n\,R^{-1}$ we get the required equality. 
\end{proof}

Henceforth, $M $ denotes the incidence matrix  of the morphism $\varphi = \varphi_{2d+1}$, i.e., 
\begin{equation}\label{def:morphism_matrix}
M=
\left(\begin{smallmatrix}
1&1&\cdots &1&1&1\\
1&0&\cdots &0&0&0\\
0&1&\cdots &0&0&0\\
\vdots&&\ddots&\ddots&&\vdots\\
&&&&&\\
0&0&\cdots &1&0&0\\
0&0&\cdots &0&1&2\\
\end{smallmatrix}\right)\in \mathbb{R}^{(2d+1)\times (2d+1) }\,.
\end{equation}
We make use of the following properties of $M$ that are easy to check. 

 \begin{description}
     \item[P1] 
 The characteristic polynomial of $M$ equals
$$\chi(t) = \underbrace{(t^{d+1}-2t^{d}-1)}_{=:f(t)}\underbrace{(t^d-t^{d-1} -t^{d-2}-\cdots - t-1)}_{=:g(t)} = \tfrac{1}{t-1}\Bigl(t^{2d}\bigl(t-2)^2-1\Bigr).$$

Note that the trinomial $f$ is irreducible by \cite{Har2012} and   satisfies the Hollander condition~\cite{Hol1996} (a polynomial $h(t)= t^n - a_{n-1}t^{n- 1}-a_{n-2}t^{n-2}- \cdots - a_0$ with  $a_i \in \N, a_0\neq 0$, satisfies the Hollander condition if  $a_{n-1}> a_{n-2}+a_{n-3}+\cdots +a_0$)
and thus the  dominant root of $f$, denote it  $\Lambda$,  is a Pisot number, i.e., $\Lambda >1$ and all other roots of $f$ are in modulus
 strictly smaller than $1$.  Similarly,  the polynomial $g$ satisfies the Frougny-Solomyak condition~\cite{FrSo1992}  (a polynomial $h(t)= t^n - a_{n-1}t^{n- 1}-a_{n-2}t^{n-2}- \cdots - a_0$  with  $a_i \in \N, a_0\neq 0$, satisfies the Frougny-Solomyak  condition if  $ a_{n-1}\geq a_{n-2}\geq \cdots \geq a_0$). Hence $g$ is irreducible and  the dominant root, denoted $\beta$, is also Pisot.  In summary, all eigenvalues in the spectrum of $M$ are distinct and the spectrum contains just two eigenvalues outside the closure of  the unit circle, namely 
$\Lambda \in (2,3)$ and $\beta \in (1,2)$. They   satisfy   \begin{equation}\label{eq:rovnice}\Lambda^{d+1}-2\Lambda^d=1\qquad \text{and } \qquad \beta^{d+1}-2\beta^d=-1.\end{equation}
Denote $S_\Lambda$ the set of the algebraic conjugates of $\Lambda$ and by $S_\beta$ the set of the algebraic conjugates of $\beta$. The spectrum of $M$ equals $S_\Lambda \cup S_\beta$ and the elements of the spectrum are mutually distinct. 
\item[P2] The right eigenvector of $M$ corresponding to an eigenvalue $\lambda \in S_\Lambda \cup S_\beta $ is 
$$
\vec{v}_\lambda = \bigl(\lambda^{2d-1}(\lambda-2),\lambda^{2d-2}(\lambda-2), \ldots , \lambda(\lambda-2), \lambda-2, 1\bigr)^T;$$
the left row eigenvector to $\lambda$ is 

$$
\vec{u}_\lambda = \bigl(1, \lambda -1, \lambda^2 - \lambda -1, \ldots, \lambda^{2d} - \lambda^{2d-1} - \cdots - \lambda -1\bigr),$$
equivalently, the  $i^{th}$ component  of $\vec{u}_\lambda$ is $\frac{1}{\lambda-1}\bigl( \lambda^{i+1} - 2\lambda^i +1\bigr)$ for $i=0,1,\ldots, 2d$.

In particular,  the last  entry of  $\vec{u}_\lambda$ is 
$$\frac{1}{\lambda -1}\bigl(\lambda^{2d+1} - 2 \lambda^{2d} +1\bigr) = \left\{ \begin{array}{cc}\lambda^d &\text{ if $\lambda \in S_\Lambda$\,;}\\
\frac{-\lambda^d +1}{\lambda -1} & \text{ if $\lambda \in S_\beta$\,.}
\end{array}\right.$$

\item[P3] If  $\lambda \in S_\Lambda$, then 
$$  \vec{u}_{\lambda} \vec{v}_\lambda = \frac{2\lambda^{d-1}}{\lambda -1}\Bigl(d(\lambda-2)+ \lambda\Bigr). $$

 \end{description}

 \begin{lemma}\label{lem:ReseniRekurence} Let $n \in \N$ and  $\vec{h} =(\underbrace{1,1, \ldots,1}_{d\text{-items}}, \,\underbrace{2^0,2^1, \ldots, 2^{d}}_{(d+1)\text{-items}} \,).$  Then \  
$$
\vec{h}\,M^n \,  \vec{e}_{2d} = \sum\limits_{\lambda \in S_\Lambda} c_\lambda \lambda^n, \  \text{where } \ c_\lambda = \frac{1}{(d+1)\lambda - 2d}\, \lambda^{d+1}.   
$$
\end{lemma}
\begin{proof} 
By P1,  we can apply  Lemma \ref{lem:obecnaM}.  Using P2 and \eqref{eq:rovnice} we derive 
$$ \vec{h}\, \vec{v}_\lambda = \left\{
\begin{array}{ll}
    \frac{2\lambda^{d}}{\lambda -1}\quad  &  \text{if } \lambda \in S_\Lambda;\\
    0 & \text{if } \lambda \in S_\beta,
\end{array} \right.
$$
and $ \vec{u}_\lambda \,\vec{e}_{2d} =\lambda^d$ for $\lambda \in S_\Lambda$. Exploiting P3, we 
get  
$$\vec{h}\,M^n \, \vec{e}_{2d} = \sum_{\lambda \in S_\Lambda} \tfrac{1}{\vec{u}_\lambda\,\vec{v}_\lambda}\, \tfrac{2\lambda^d}{\lambda -1} \lambda^d \, \lambda^n =  \sum_{\lambda \in S_\Lambda} \tfrac{\lambda^{d+1}}{(d+1)\lambda - 2d}\lambda^{n} .
$$

\end{proof}


\begin{thebibliography}{30}
\bibitem{BaMaPe2007} P. Baláži, Z. Masáková, E. Pelantová, Factor versus palindromic complexity of uniformly recurrent infinite words, Theoret. Comput. Sci. 380 (2007), 266--275.

\bibitem{BaPeSt2010} Ľ. Balková, E. Pelantová, Š. Starosta, Sturmian jungle (or garden?) on multiliteral alphabets, RAIRO – Theor. Inform. Appl. 44 (2010), 443--470.

\bibitem{BaPeSt2011} Ľ. Balková, E. Pelantová, Š. Starosta, Infinite words with finite defect, Advances in Applied Mathematics 47(3) (2011), 562--574.

\bibitem{BaSh19} A. R. Baranwal, J. O. Shallit, {Repetitions in infinite palindrome-rich words}. In: Mercas, R., Reidenbach, D. (eds.) Proceedings WORDS 2019, Lecture Notes in Computer Science, vol. 11682, pp. 93--105. Springer (2019).
https://doi.org/10.1007/978-3-030-28796-2{\_}7

\bibitem{BuLuGlZa2009} M. Bucci, A. De Luca, A. Glen, L. Q. Zamboni, A connection between palindromic and factor complexity using return words, Adv. in Appl. Math. 42 (2009), 60--74.

\bibitem{Ca2008} J. Cassaigne, {On extremal properties of the Fibonacci word}. RAIRO - Theor. Inform. Appl. 42(4) (2008), 701--715.
https://doi.org/10.1051/ita:2008003

\bibitem{CuMoRa2020} J. D. Currie, L. Mol, N. Rampersad, {The repetition threshold for binary rich words},
 Disc. Math. \& Theoret. Comput. Sci. 22(1) (2020). 

\bibitem{CuMoPe2024} J. D. Currie, L. Mol, J. Peltomäki, {The repetition threshold for ternary rich words},
The electronic journal of combinatorics 32(2) (2025), \#P2.55. https://doi.org/10.37236/13499

\bibitem{DolceDP2023} 
 F. Dolce, Ľ. Dvo{\v r}{\'{a)}}kov{\'{a}}, E. Pelantov{\'{a}}, {On balanced sequences and their critical exponent},
 Theoret. Comput. Sci. 939 (2023), 18--47. https://doi.org/10.1016/j.tcs.2022.10.014

\bibitem{DrJuPi2001}
X. Droubay, J. Justin, G. Pirillo, Episturmian words and some constructions of de Luca and Rauzy, Theoret. Comput. Sci. 255 (2001), 539--553. https://doi.org/10.1016/S0304-3975(99)00320-5

\bibitem{Dur98}
{F.~Durand}, {A characterization of substitutive sequences using return words},
Discrete Math. 179(1-3) (1998), 89--101.

\bibitem{DvKlPe2025} Ľ. Dvo{\v r}{\'{a}}kov{\'{a}}, K. Klouda, E. Pelantová, The asymptotic repetition threshold of sequences rich in palindromes, Eur. J. of Combin. 126 (2025), 104124.
https://doi.org/10.1016/j.ejc.2025.104124

\bibitem{DvPe2025} Ľ. Dvo{\v r}{\'{a}}kov{\'{a}}, E. Pelantová, {Symmetries of rich sequences with minimum critical exponent}, In: Gamard, G., Leroy, J. (eds) Combinatorics on Words. WORDS 2025. Lecture Notes in Computer Science, vol 15729. Springer, Cham., 143--154.

\bibitem{Ochem} Ľ. Dvo{\v r}{\'{a}}kov{\'{a}}, P. Ochem, D. Opo\v censk\'a, {Critical exponent of binary words
with few distinct palindromes}, The electronic journal of combinatorics 31(2) (2024), \#P2.29. https://doi.org/10.37236/12574 

\bibitem{FoSaVa2025} E. Foster, A. Saarela, A. Vanhatalo, {Mapped exponent and asymptotic critical exponent of words}, preprint, 2025, available at https://arxiv.org/abs/2506.04091

\bibitem{FrSo1992} Ch. Frougny, B. Solomyak, Finite $\beta$-expansions, Ergodic Theory Dynam. Systems 12 (1992), 713--723.

\bibitem{GlJuWi2009} A. Glen, J. Justin, S. Widmer, L.Q. Zamboni, {Palindromic richness}, Eur. J. Combin. 30 (2009), 510--531.

\bibitem{GOKlSt2025} H. Goulet-Oullet, K. Klouda, Š. Starosta, {Circularity and repetitiveness in non-injective DF0L systems}, In: Gamard, G., Leroy, J. (eds) Combinatorics on Words. WORDS 2025. Lecture Notes in Computer Science, vol 15729. Springer, Cham., 153--165.


\bibitem{Har2012} J. Harrington, {On the factorization of the trinomials $x^n + cx^{n-1} + d$},  Int. J. Number Theory, 8(6) (2012), 1513-–1518. 

\bibitem{Hol1996} M. Hollander, Linear numeration systems, finite beta-expansions, and discrete spectrum of
substitution dynamical systems, PhD Thesis, Washington University (1996).

\bibitem{Klouda2012} K. Klouda, Bispecial factors in circular non-pushy D0L languages,
 Theoret. Comput. Sci. 445 (2012), 63--74.
 
 \bibitem{KlSt2015} K. Klouda, Š. Starosta, An algorithm for enumerating all infinite repetitions in a D0L-system, Journal of Discrete Algorithms 33 (2015), 130--138. https://10.1016/j.jda.2015.03.006
 
\bibitem{PStarosta2013} E. Pelantová, \v S. Starosta, Languages invariant under more symmetries: overlapping factors versus palindromic richness, Discrete Math. 313 (2013), 2432--2445. https://doi.org/10.1016/j.disc.2013.07.002 

\bibitem{Quef87} M. Queff\'elec, Substitution Dynamical Systems — Spectral Analysis, vol. 1294 of Lecture Notes in Mathematics. Springer-Verlag, 1987.








 




 












































\end{thebibliography}
\end{document}